\documentclass[reqno,11pt]{amsart} \numberwithin{equation}{section}
\input{amssym.def}
\input{amssym.tex}

\usepackage{graphicx}
\begin{document}
\title[A Domain-Theoretic Bishop-Phelps  Theorem] {A Domain-Theoretic Bishop-Phelps Theorem $^{*}$}
\author[A. Hassanzadeh, I. Sadeqi and  A. Ranjbari]{Ali Hassanzadeh$^{1}$, Ildar Sadeqi$^{2}$ and  Asghar Ranjbari$^3$}
\address{\indent $^{1,2}$Department of Mathematics,
\newline \indent Sahand University of Technology, Tabriz, Iran.
}
\email{\rm $^1$ a\_hassanzadeh@sut.ac.ir, ali.hassanzadeh@guest.unimi.it}
\email{\rm $^2$ esadeqi@sut.ac.ir}
\address{\indent $^{3}$Department of Pure Mathematics, Faculty of
Mathematical Sciences,
\newline \indent  University of Tabriz,
 Tabriz, Iran.} \email{\rm ranjbari@tabrizu.ac.ir}
\thanks{*Corrected version}
\begin{abstract}
In this paper, the notion of  $c$-support points of a set in a
semitopological cone is introduced. It is  shown  that  any nonempty
convex Scott closed bounded set  has a $c$-support point in a
cancellative   $bd$-cone under certain condition.
We also introduce the notion of $wd$-cone  and then we prove a Bishop-Phelps type theorem for  $wd$-cones, especially for normed cones, under appropriate conditions. Finally, using of the Bishop-Phelps  technique,
 we obtain  a result about the fixed points of a mapping on $s$-cones.

\end{abstract}


\maketitle

\baselineskip=15.8pt

  \newtheorem{thm}{Theorem}[section]
\newtheorem{lem}[thm]{Lemma}
\newtheorem{prop}[thm]{Proposition}
\newtheorem{cor}[thm]{Corollary}
\newtheorem{conj}[thm]{Conjecture}
\theoremstyle{definition}
\newtheorem{deff}[thm]{Definition}
\newtheorem{exm}[thm]{Example}
\theoremstyle{remark}
\newtheorem{rem}[thm]{Remark}
\setcounter{section}{0}

 \newcommand{\bd}[1]{\mathbf{#1}}  
\newcommand{\RR}{\mathbb{R}}      
\newcommand{\ZZ}{\mathbb{Z}}      

\newcommand\twoheaduparrow{\mathrel{\rotatebox[origin=c]{90}{$\twoheadrightarrow$}}}
\newcommand\twoheaddownarrow{\mathrel{\rotatebox[origin=c]{90}{$\twoheadleftarrow$}}}

\section{Introduction}
Domain theory which is based on logic and computer science, started as an outgrowth of theories
of order. Progress in this domain rapidly required a lot of material on
(non-Hausdorff) topologies.
After about 40 years of domain theory, one is forced to recognize that
topology and domain theory have been beneficial to each other \cite{hofmann, jean}.

One of Klaus Keimel's many mathematical interests is the interaction
between order theory and functional analysis. In recent years this
has led to the beginnings of a ‘domain-theoretic functional
analysis, which may be considered to be a topic within
positive analysis in the sense of Jimmie Lawson \cite{lawson}.
In the latter, notions of positivity and order play a key role,
as do lower semicontinuity and so $T_0$ spaces.
 Some basic functional analytic tools were developed by Roth and Tix  and later
 by Plotkin and Keimel   for these structures.
Roth has written  several papers in this area including his  papers
 \cite{roth, roth2002}
 on Hahn-Banach type theorems for locally convex cones.  Tix  in her 1999 Ph.D. thesis
 gave a domain-theoretic
version of these theorems in the framework of $d$-cones (see  \cite{ tix3, tix}).
Plotkin subsequently gave another separation theorem, which was
incorporated, together with other improvements, into a revised
version of Tix's thesis \cite{keimel3,plotkin}. Finally,   Keimel
\cite{keimel} improved the Hahn-Banach theorems to semitopological
cones.

The theory of locally convex cones, with applications to Korovkin
type approximation theory for positive operators and to
vector-measure theory, was developed in the books by Keimel and Roth
\cite{109} and Roth \cite{212}, respectively.

The Bishop-Phelps theorem \cite{bishop} is a fundamental theorem in
functional analysis which has many applications in the geometry of
Banach spaces, fixed point theory and optimization (for instance see
\cite{fixedbook, fabian}).
The classical Bishop-Phelps theorem states that ``the set of support
functionals for a closed bounded convex subset $B$ of a real Banach
space $X$ is norm dense in $X^*$ and the set of support points of
$B$ is dense in the boundary of $B$''
 \cite{bishop}.
The present paper contributes a domain-theoretic analogue
of the classical Bishop-Phelps theorem for semitopological cone.

The work on Hahn-Banach-type theorems has found application in theoretical
computer science, $viz.$ the study of powerdomains.
 It was a pleasant surprise that the separation theorems found application in this
development and we anticipate that so too will the domain-theoretic Bishop-Phelps theorem given here.
As an application of the  Bishop-Phelps  theorem, we show that a mapping on a $wd$-cone has a fixed point under some conditions.

\section{Preliminaries}
\label{sectionp}
For convenience of the reader we give a survey of the relevant materials from  {\cite{abram}}, {\cite{Alip}}, {\cite{jean}} and {\cite{keimel}},  without
proofs, thus making our exposition self-contained.

Let $B$ be a nonempty subset of a real Banach space $X$ and $f$ be
a nonzero continuous linear functional on $X$. If $f$
attains either its maximum or its minimum over $B$ at
the point $x \in B$, we say that $f$ supports $B$ at $x$ and
that $x$ is a {\it support point} of $B$.

For subsets $A$ of a partially ordered set $P$ we use the following notations:

\centerline{$\downarrow A =: \{x \in P | x \leq a  \text{ for  some} \  a \in A\},$}

\centerline{$\uparrow A =: \{x \in P | x \geq a \ \text{for some}  \ a \in A\}.$}

It is called  that $A$ is a lower or upper set, if $\downarrow A = A$ or $\uparrow A = A$, respectively.

We denote by $\Bbb R_+$  the subset of all nonnegative reals.
Further, $\overline{\Bbb R} = \Bbb R \cup \{+\infty\}$and $
\overline{\Bbb R}_+ = \Bbb R_+ \cup \{+\infty\}$. Addition,
multiplication and the order are extended to $+\infty$ in the usual way. In particular, $+\infty$
becomes the greatest element and we put $0 \cdot (+\infty) = 0$.

According to \cite{keimel},
 a {\it cone} is a set $C$, together with two
operations $+ : C \times C \rightarrow$ C and $\cdot :\Bbb R_+\times C \rightarrow C$ and
 a neutral element $0 \in C $, satisfying the
following laws for all $v,w, u \in C$ and $\lambda, \mu \in \Bbb R_+$:
\begin{align*}
    &0+v=v,&&1v=v,\\
  &v+(w+u) = (v+w)+u,&& (\lambda \mu)v = \lambda(\mu v),\\
  &v+w=w+v,&&(\lambda + \mu)v = \lambda v + \mu v, \\
  & &&\lambda(v+w) = \lambda v + \lambda w.
\end{align*}

An {\it ordered cone} $C$ is a cone endowed with a partial order $\leq$ such that the
addition and multiplication by
fixed scalars $r \in \Bbb R_+$ are order preserving, that is, for all $x, y, z \in~C$ and all $r \in\Bbb R_+$:
$$x \leq y \Rightarrow x + z \leq y + z\ \  \text{and}\ \  r · x \leq r · y. $$

Let us recall that a {\it{linear function}} from a cone $(C,+,
\cdot)$ to a cone $(C',+, \cdot)$ is a function $f:C \rightarrow C'$
such that $f(v + w) = f(v) + f(w)$ and $f(\lambda v) = \lambda
f(v)$, for all $v,w \in C$ and $ \lambda \in \Bbb R_+$.

A subset $D$ of a cone $C$ is said to be {\it{convex}} if for all $u, v \in D$ and
 $\lambda \in [0, 1], \lambda u + (1 - \lambda)v \in D$. The convex hull of a
  set $D$ is defined to be the smallest convex set containing $D$.

For example, $(\Bbb R_+)^n$ is a cone, with respect to the coordinate-wise
operations. On $\Bbb R_+$, the cone order is just the usual order
$\leq$ of the reals. On $(\Bbb R_+)^n$, it is the coordinate-wise
order.

Recall that a partially ordered set $(A,  \leq)$  is called directed if for every $a , b \in A$ there exists $c\in A$ with $a,b \leq c$.
 A partially ordered set  $(D, \leq)$ is called a directed complete partial order ({\it dcpo}) if every directed subset $A$ of $D$, has a
least upper bound in $D$. The least upper bound of a directed subset $A$ is denoted by $\sqcup^\uparrow A$,
and it is also called the directed supremum.

In any partially ordered set $P$, the {\it way-below} relation $x \ll y$
is defined by: $x \ll y$ iff, for any directed subset $D \subset P$
for which supremum of $D$ exists, the relation $y \leq \sqcup^\uparrow D$ implies the
existence of a $d \in D$ with $x \leq d$.
 An element $y \in P$ is called finite if, $y \ll y$.

 The partially ordered set
$P$ is called continuous if, for every element $y$ in $P$, the set
$\twoheaddownarrow y=:~\{x\in~P; x \ll y\}$ is directed and
 $y =\sqcup^\uparrow \twoheaddownarrow y$.
 Note that $x \ll y$ implies $x \leq y$ {\cite[Prop. 5.1.4]{jean}}.

Any $T_0$-space $X$ comes with an intrinsic order, the {\it specialization order} which is defined by $x \leq
y$ if the element $x$ is contained in the closure of the singleton $\{y\}$ or, equivalently, if every open
set containing $x$ also contains $y$.

Given any  ordering $\leq$, there are at least two topologies
with $\leq$ as specialization  ordering, the coarsest possible one (the upper topology) and the
finest possible one (the Alexandroff topology) (see {\cite[Sec. 4.2.2]{jean}} for more details). Additionally, there are some other interesting topologies in
between. An important example of a topology that sits in between is the Scott topology.

Let $D$ be a partially ordered set. A subset $A$ is called
 Scott closed if it is a lower set and is closed under supremum of
directed subsets, as far as these suprema exist. Complements of
Scott closed sets are called  Scott open.
The collection of all Scott open sets is a topology, called the {\it Scott topology} on $D$ {\cite[Prop. 4.2.18]{jean}}. We
write $D_{\sigma}$ for the set $D$ with the Scott topology.

The basic notion is that of a {\it Scott continuous} function: A
function $f$ from a partially ordered set $P$ to a partially ordered
set $Q$ is called Scott continuous if it is order preserving and if,
for every directed subset $D$ of $P$ which has a least upper bound
in $P$, the image $f (D)$ has a least upper bound in $Q$ and $f
(\sqcup^\uparrow D) = \sqcup^\uparrow f(D)$.

Let $P, Q$ be two partially ordered sets. A map
$f : P_{\sigma}\rightarrow Q_{\sigma}$ is continuous iff $f : P \rightarrow Q$ is Scott continuous
{\cite[Prop. 4.3.5]{jean}}.

In a continuous partially ordered set $C$, the set $\twoheaduparrow
x$ is Scott open for all $x$. More generally, for every subset $E$ of $C$,
the subset $\twoheaduparrow E$ is open in $\sigma_C$
 {\cite[Prop. 5.1.16]{jean}},
 so $\twoheaduparrow E \subset int(\uparrow E)$.
If   the subset $E\subset C$ is finite, then
$\twoheaduparrow E = int (\uparrow E)$
{\cite[Prop. 5.1.35]{jean}}.

On the extended reals $\overline{\Bbb R}$ and on its subsets ${\Bbb
R_+}$ and $\overline{\Bbb R}_+$ we use the {\it upper topology},
the only open sets for which are the open intervals  $\{s : s > r\}$.
This upper topology is $T_0$, but far from being Hausdorff.
\subsection{Semitopological Cones}
According to \cite{keimel}, a {\it semitopological cone} is a  cone
with a $T_0$-topology such that the addition and scalar multiplication
are separately continuous, that is:
\begin{align*}
&a \mapsto ra: C \rightarrow C& \mbox {is continuous  for  every fixed} \ r > 0,& \\
&r \mapsto ra: \Bbb R_+ \rightarrow C& \mbox{is continuous for every fixed} \ a \in C,&  \\
& b \mapsto a + b : C \rightarrow C& \mbox{is continuous for every
fixed} \ a \in C.&
\end{align*}

An {\it $s$-cone} is a cone with a partial order such that addition and scalar multiplication:
$$(a,b) \mapsto a + b : C\times C \rightarrow C, \ \ (r, a) \mapsto ra:\Bbb R_+\times C \rightarrow C$$
are Scott continuous. A $s$-cone is called a {\it $[b]d$-cone} if its order is [bounded] directed
complete, i.e., if each [upper bounded] directed subset has a least upper bound.

Note that every $s$-cone is a semitopological cone with respect to
its Scott topology {\cite[Prop. 6.3]{keimel}}.

A cone $C$ with a topology is called {\it locally convex}, if each point has a
neighborhood basis of open convex neighborhoods.

 Let $C$ be a semitopological cone. The cone $C^*$ of all linear continuous
functionals $f: C \rightarrow \overline{\Bbb R}_+$ are called
dual of $C$.

We shall use the following {\it separation theorem}
{\cite[Theorem 9.1]{keimel}}:
in a semitopological cone $C$ consider a nonempty convex subset $A$ and
an open convex subset $U$. If $A$ and $U$ are disjoint, then there exists a continuous linear
functional $f: C \rightarrow \overline{\Bbb R}_+$ such that $f(a) \leq 1 < f(u)$ for all $a \in A$ and all $u \in U$.

Finally, we shall use the following {\it  strict separation theorem}
{\cite[Theorem 10.5]{keimel}}:
let C be a locally convex semitopological cone.
Suppose that $K$ is a compact convex set and that $A$ is a nonempty closed convex set
disjoint from $K$. Then there is a continuous linear functional $f$ and an $r$
such that $f (b) \geq r >1 \geq f (a)$ for all $b$ in $K$ and all $a$ in $A$.
\subsection{Normed Cones}
A cancellative cone (more precisely cancellative asymmetric  cone)  is a cone $C$,  satisfying
the following laws for all $v,w, u \in C$:
\begin{align*}
   &v+u = w+u \Rightarrow v=w,&&(cancellation)\\
  &v+w = 0 \Rightarrow v=w=0.&& (strictness)
\end{align*}
Let $C$ be a cancellative cone, we define a partial order on $C$ by $x\preccurlyeq y \Leftrightarrow y \in x+C$,  called  the {\it cone order} on $C$.

According to \cite{selinger}, a  norm on  a  cancellative cone $C$ is a function $\parallel \cdot \parallel : C \rightarrow \mathbb{R}_+$
satisfying the following conditions for all $v,w \in C$ and $\lambda \in \mathbb{R}_+:$
\begin{align*}
    &\parallel v+w \parallel \leq \parallel v\parallel + \parallel w\parallel ,\\
  &\parallel\lambda v \parallel = \lambda \parallel v\parallel ,\\
  &\parallel v\parallel= 0 \Rightarrow v=0 ,\\
  &v \preccurlyeq w \Rightarrow \parallel v\parallel \leq \parallel w\parallel .
\end{align*}

A {\it normed cone} $C = \langle C, \parallel\cdot\parallel \rangle$ is a  cancellative cone equipped with a norm.
 The {\it unit ideal} of a normed cone $C$ is the set
$$U_C = \{ u \in C ; \parallel u \parallel \leq 1\}.$$

A normed cone $C$ is called {\it complete} if its unit ideal
is a dcpo.
For example
the   normed cones $\mathbb{R}_+, \mathbb{R}^n_+$, $l_{\infty}^+$ (the set of all bounded sequences in $\mathbb{R}_+$) together with the supremum norm $\|(x_i)_i\| = \sup_i x_i$ and $l_1^+$ (the set of all sequences in $\mathbb{R}_+$ of bounded sum) together with the sum
norm $\|(x_i)_i\| = \sum_i x_i$ are all complete and continuous {\cite[Exam. 2.7]{selinger}}.
We will say simply   continuous normed cone for continuous complete normed cone.

\section{Main results}
The purpose of this section is to establish the Bishop-Phelps type
theorem for semitopological cones. Indeed we want to study the Bishop-Phelps theorem  in non-Hausdorff setting.

\begin{rem}
($a_1$) Let $B$  be a nonempty  Scott closed set in a semitopological cone $C$. Since $0\in B$, so for any linear functionals $f: C \rightarrow \overline{\Bbb R}_+$ we have $f(0) = \inf f(B)$.

 ($a_2$) If $B$ is a  nonempty  compact set in a semitopological cone $C$  and $f: C \rightarrow \overline{\Bbb R}_+$ is a continuous map, then
 there is an element $z \in B$ such that $f(z) = \inf f(B)$ {\cite[Lemma 3.8]{jean2}}.
 Since in  a semitopological cone, a compact set is not necessarily closed, so  the proof of this statement is different from the  method of the classical analysis and the  result is not true for supremum (for details see \cite{jean2}).
 \end{rem}

 Note that a subset $B$ in a partially ordered set $C$,  is called {\it bounded} if there exists an element $d\in C$ such that for any $b\in B$, $b\leq d$.

 Let $C$ be an $s$-cone. We will say that $C$ has the {\it additive  property},  if the following axioms are satisfied:

(i) $x'\ll x$ and $y'\ll y$ implies $x'+y'\ll x+y$.

(ii) $x\ll \lambda x$ for any scalar $\lambda >1$ and $x>0$.

\begin{exm}
($b_1$) $\Bbb R^n_+$   is a cancellative continuous $bd$-cone  with the ordering:
$$(x_1,...,x_n) \leq (y_1,...,y_n) \Longleftrightarrow \forall n  \ \ x_n \leq y_n.$$
In $\Bbb R^n_+$, we have
$(x_1,...,x_n) \ll (y_1,...,y_n)$  iff for all $i$, $x_i = 0$ or $x_i < y_i$.
It is easy to see that  $\Bbb R^n_+$ has the additive property.

($b_2$) The cones $\ell_1^+$ and $\ell_{\infty}^+$
are cancellative continuous $bd$-cones
 under  usual pointwise ordering which also have the additive property.
\end{exm}

It is known that the interior of a convex set  in a topological linear space is a convex set, but this is not true in semitopological cones in general \cite{keimel}.
\begin{rem}
($c_1$) In a continuous $s$-cone $C$, which has the additive property, the interior of every convex upper set is convex {\cite[Lemmas 4.10 \& 6.14]{keimel}}.

($c_2$) In a continuous normed cone $C$, for any convex
set $B$, the open set $\twoheaduparrow B$, is    convex
 {\cite[Lemma 2.16]{selinger}}.

($c_3$)
In a continuous $s$-cone $C$, which has the additive property, the interior of every  upper set is nonempty.
To see this,
let $A$ be an upper set. For $x\in A$  consider $\twoheaduparrow x$, which is a nonempty open set in $A$, so $int(A) $ is nonempty.
\end{rem}

\begin{prop}
Let $C$ be a  continuous $s$-cone which has the additive property, and $B\subset C$ be an  upper convex set.
 If $x\in  B$  such that $\lambda x \not\in B$, whenever $\lambda < 1$,
 then there exists a continuous linear functional $f: C \rightarrow \overline{\Bbb R}_+$ such that  $f(x) = \inf f (B)$.
\end{prop}
\begin{proof}
 Let $x\in B$.
  By  the assumption,   $\lambda x\not \in B$ for every $0<\lambda<1$,
 so $x\not \in int(B)$.
 The interior of $B$ is a nonempty convex open set,
therefore by the separation theorem, there exists a continuous linear functional $f: C \rightarrow \overline{\Bbb R}_+$ such that
    $f(x) \leq f (b)$ for  all $b\in int(B)$.
    By continuity of C, for each $y \in B$, $\twoheaduparrow y =: \{a : y\ll  a \}$ is an open set in $B$ and so  it is included  in $int(B)$, so $f(x) \leq f(a)$
    for  all $a\in \twoheaduparrow y$.
    By the additive property $f(x) \leq f(y)$ for each $y\in B$, hence
$f(x) = \inf f (B)$.
\end{proof}
In the sequel, we consider  suprema instead of infima.
Let $B$  be a convex  closed set in a semitopological cone $C$.
A point $x\in B$ is called a {\it c-support point} for $B$, if there
exists a linear  continuous functional $f : C \rightarrow
\overline{\Bbb R}_+$ such that $f(x) = \sup f(B)$ and $f(x) < \infty$; such
a functional  $f$ is said a {\it c-support functional}.

\begin{rem}
Let $B$  be a convex closed set in a semitopological cone $C$. Then we have the following facts:
\begin{itemize}
\item[($d_1$)]
 If  the set  $B$ has a maximum, then any  linear continuous functional $f : C \rightarrow \overline{\Bbb R}_+$ is a c-support functional for $B$.

\item[($d_2$)]
If  the set $B$ has  nonempty interior, then $B$ is an unbounded set, so any  linear  continuous functional $f : C \rightarrow \overline{\mathbb{R}}_+$ on $B$ is  unbounded.

\item[($d_3$)]
If $C$ is a $d$-cone and $B$ is a directed Scott closed set, then
$\sqcup^\uparrow B \in B$, and so every linear
Scott continuous functional is a c-support functional for $B$.
\end{itemize}
\end{rem}

Now we  restrict our attention  to the case that  $B$ is  a nonempty
convex  closed set with empty interior.
 To establish  the Bishop-Phelps theorem for semitopological cones, we  need a discussion of certain cones:

Let $C$ be a cancellative  semitopological cone and $f: C \rightarrow \overline{\Bbb R}_+$ be a  continuous linear functional. For  $0<\delta <1$ and $d\in C$, we define
$$K(f,\delta,d) = \{x\in C : f(x) < \infty \ \text{and} \  \delta x \leq f(x)d\}.$$

Note that $K(f,\delta,d)$ is a convex subcone of $C$.
Since $C$ is a cancellative cone, the  order  $ x \sqsubseteq y \Leftrightarrow y\in ~x + K$, defines a partial order on $C$,
which is called  the {\it subcone order} on $C$.
If $ x \sqsubseteq y $, then we sometimes  write $x-y$ for the unique element $z$ such that $x+z = y$.
\begin{lem}
\label{lemk}
Let $C$ be a cancellative  semitopological cone. Then for every $x,y \in C$ we have
$$ x \sqsubseteq y \ (  y\in x+K) \Rightarrow x\leq y\  (\text{with the specialization order}). $$
\end{lem}

\begin{proof}
Let $ x \sqsubseteq y$. For some $z\in K$, $y= x+z$. By the
definition of  semitopological cone,
  we know that the function $S:b \mapsto x + b : C \rightarrow C$ is continuous. So $S(\overline{\{z\}}) \subset \overline{S(z)}$ and then $x\in \overline{\{y\}}$
  and so
   $x\leq y$.
\end{proof}

Now we investigate the first part of the Bishop-Phelps type theorem for $bd$-cones.

Let $X$ be a partially ordered set, with ordering $\leq$. The specialization ordering of the Scott topology is the original ordering $\leq$
{\cite[Prop. 4.2.18]{jean}}.
\begin{rem}
Note that in a  continuous cancellative semitopological cone $C$, we have $x\not \in \twoheaduparrow~x $,  in fact there is no finite element in $C$.
Because, if $x \in \twoheaduparrow x $, since $C$ is a continuous cone, so $B :=\twoheaduparrow x$ is an open set.
With considering  the directed set $A :=  \{{\lambda} x \ : \ \lambda\in \Bbb R_+,\lambda<1\}$ which its supremum is in $B$. Since $C$ is cancellative, it follows that $A\cap B = \varnothing$, and this contrary to opening of $\twoheaduparrow x$.
\end{rem}
 \begin{thm}
Let $B$ ($\neq \{0\}$) be a nonempty convex dcpo
 in a continuous
cancellative $bd$-cone $C$, where  $C$ has the additive property.
Then there exists   some $m\in B$ such that
$B\cap (m+C) =\{m\}$. Such an
 $m$ is also a c-support point for $B$.
 \end{thm}

 \begin{proof}
 It is easy to see that the partially ordered set $(B,  \leq)$ has a maximal element, in fact
 by Zorn's lemma, it suffices to prove that every  chain in $(B, \leq)$  has an upper bound in $B$.

  By Lemma \ref{lemk} $(B, \preccurlyeq)$ (with the cone order)  has a maximal element, say $m$.
   It follows that $B\cap \uparrow m = \{m\}$ or $B\cap~(m+~C) =~\{m\}$. Since $C$ is continuous and has the additive property, so $\twoheaduparrow m$ is  a nonempty convex  open set, hense  $m \not \in \twoheaduparrow m$ and
 $B\cap( \twoheaduparrow m) = \varnothing$. By the separation theorem
 there exists a Scott continuous linear functional $f: C \rightarrow \overline{\Bbb R}_+$ satisfying $f(b) \leq f(y)$ for all $b\in B$ and  $y\in \twoheaduparrow m$. So $f(b) \leq f(\lambda m)$ for all $b\in B$ and  $\lambda >1$. Therefore, $m$ is a c-support point.
 \end{proof}
\subsection{$wd$-Cones}
In this section we  introduce  and study the notion of  $wd$-cones.
\begin{deff}
Let $C$ be a semitopological cone. The net
$\{x_\alpha\}$ is  Cauchy if there exists $0<d\in C$, satisfying the condition that, for any $\epsilon > 0$ there exists
$\alpha_0$,  such that for $\alpha, \beta \geq \alpha_0$,
$x_\alpha \leq x_\beta + \epsilon d $ and $ x_\beta \leq x_\alpha +
\epsilon d$.

\end{deff}

In the sequel, by  an order on a semitopological cone
we will always  mean the specialization order $\leq$, if not specified otherwise.

\begin{deff}
A $s$-cone  $C$ is called a $wd$-cone,  if each  increasing
Cauchy net, has a least upper bound.
\end{deff}
Clearly every $bd$-cone is a $wd$-cone. The following example shows that the converse does not hold in general.
\begin{exm}
 Let $C^+[0, 1]$ denote the cone of all continuous functions $f: [0, 1]\rightarrow \Bbb R_+$, which is also an ordered cone under
  the usual pointwise ordering.
Note  that  $C^+[0, 1]$ is not a
   $bd$-cone. To see this, consider the sequence of piecewise linear function in $C^+[0, 1]$ defined by
    \begin{equation*}
f_n(x) = \left\{
\begin{array}{cl}
1 & \text{if }0\leq x \leq {1\over 2} - {1\over n}, \\
-n(x-{1\over 2}) & \text{if } {1\over 2} - {1\over n}< x <{1\over 2}, \\
x & \text{if } {1\over 2}\leq x\leq 1.
\end{array} \right.
\end{equation*}
Thus $0\leq f_n \leq \textbf{1}$ in $C^+[0, 1]$ and  is an increasing sequence, where $\textbf{1}$ is the constant function one, but $\{f_n\}$ does not have a supremum in  $C^+[0, 1]$.
It is easy to see that  the functions
    $(f,g) \mapsto f + g : C^+[0, 1]\times C^+[0, 1] \rightarrow C^+[0, 1]\ \ \text{and} \ \  (r, f) \mapsto rf:\Bbb R_+\times C^+[0, 1] \rightarrow C^+[0, 1]$
are Scott continuous. So $C^+[0, 1]$ is an $s$-cone. Furthermore,  $C^+[0, 1]$  is a continuous $s$-cone.
Now we show that
 $C^+[0, 1]$ is a $wd$-cone. Indeed, if $\{f_{\alpha}\}$ is an  increasing Cauchy net in $C^+[0, 1]$, then
     it is a norm Cauchy net. Since $C[0, 1]$ is a Banach space,  the net $\{f_{\alpha}\}$  is norm-convergent  to some $f$. This means that
 $f_{\alpha}(x) \rightarrow  f(x)$, Thus $f(x) = \sup_{\alpha} f_{\alpha}(x) $ for each $x\in [0,1]$. It follows that $\sup_{\alpha} f_{\alpha} = f$ and $f\in C^+[0, 1]$.
\end{exm}
In the sequel, the mean of a $wd$-cone will always a cancellative continuous  $wd$-cone, if not specified otherwise.

\subsection{A Bishop-Phelps type Theorem}
In this section, we prove the Bishop-Phelps type theorem for $wd$-cones.
\begin{lem}
\label{lem1}
Let  $f$ be a  continuous linear functional on a    $wd$-cone $C$,  $0<\delta <1$
 and $d\in C$. If $B$  is a nonempty convex bounded  closed  subset of $C$,
  then for each $b\in B$ there exists  a maximal element $m\in B$ satisfying $B\cap (m+K(f,\delta ,d)) = \{m\}$ and $b \sqsubseteq m$.
\end{lem}
\begin{proof}
It is sufficient to show that  the set $B_b = \{y\in B : b\sqsubseteq y\}  $  with the order $\sqsubseteq$ has a maximal element.

By Zorn's Lemma, it suffices to prove that every chain in $(B_b,  \sqsubseteq)$  has an upper bound in $B_b$. Let $Z$ be a chain in $B_b$. If we let  $x_{\alpha} = \alpha $ for each $\alpha \in Z$, we can identify $Z$  with the increasing net $\{ x_{\alpha}\}$.

Without loss of generality, let $x_{\alpha}$ and $x_{\beta}$ be two elements of the net. Without
lose of the generality, we can suppose that $x_\alpha \sqsubseteq
x_\beta$. So there exists $k \in K$ such that $x_\beta = x_\alpha +k$
and $\delta k \leq f(k) d$. Therefore,
 $\delta x_\beta \leq \delta x_\alpha +(f(x_\beta) - f(x_\alpha))d$. By the boundedness  of B and  continuity of $f$, it follows that $f(x_\alpha) $
 is a bounded net, and so
    $f(x_\alpha)$ is  convergent and Cauchy. It is easy to see that the  net $\{x_\alpha\}$ is  a directed
     Cauchy  and so has a supremum, say  $x$ (with specialization order), that means $\sup_{\leq} x_\alpha = x$.
    Since $B$ is  a Scott closed set, so $x\in B$.
     Now fix $\beta \in Z$, so we have $x_{\beta}\sqsubseteq x_{\alpha}$ for $\beta \sqsubseteq \alpha$. Thus
     $\delta x_{\alpha} \leq (f(x_{\alpha})-f(x_{\beta})) d+\delta x_{\beta}$, which follows that
       $\delta x \leq (f(x)-f(x_{\beta})) d+\delta x_{\beta}$. Hence $x_{\beta} \sqsubseteq x$ and
      $x$ is an upper bound of the net $(\{x_\alpha\},\sqsubseteq)$.
      It follows that $x\in B_b$,  hence ($B_b , \sqsubseteq$)  has a maximal element; say $m$.  Therefore $B \cap (m + K) = \{m\}$ and $b \sqsubseteq m$.
\end{proof}

\begin{lem}
\label{lem2}
Let  $C$ be  a cancellative continuous  $s$-cone and let $K$ be a subcone of $C$. If $B$  is a nonempty  subset of $C$  and $0\neq m\in B$, such that  $B\cap (m+K) = \{m\}$, then
 $B\cap\uparrow~ (m+K\verb|\| \{0\}) =\varnothing$. In particular,
  $B\cap int(\uparrow (m+K\verb|\| \{0\})) =\varnothing$, $B\cap \twoheaduparrow (m+K\verb|\| \{0\}) =\varnothing$.
\end{lem}
\begin{proof}
Let $x\in B\cap \uparrow (m+K\verb|\| \{0\})$, then there
exists a $k\in K\verb|\|\{0\}$ suth that  $x=m+k $. By the assumption, $m+k = m$, so $k=0$. This leads to a
contradiction.
\end{proof}
Applying the separation theorem
  and Lemmas \ref{lem1} and \ref{lem2},
   we obtain the  following  Bishop-Phelps type theorem for $wd$-cones, the main result of this paper.
\begin{thm}
\label{asli}
Let $B$  be a nonempty convex bounded   closed  set in a  locally convex $wd$-cone $C$, such that $C$ has the additive property. Then we have:

 ($e_1$) Fix $\epsilon >0$ and $d\in C$. For each $x_0\in B$, such that $\lambda x_0 \not \in B$ whenever $\lambda >1$,  there exist a  continuous linear functional $f: C \rightarrow \overline{\Bbb R}_+$  and  an $m\in B$ such that  $f(m) = \sup f (B)$ and $x_0\leq m\leq x_0 + \epsilon  d$.

($e_2$) For each  continuous linear functional $f: C \rightarrow \overline{\Bbb R}_+$,  there exists a c-support functional $h$ for $B$  such  that $0\leq h \leq  f$ on a subcone of $C$.
\end{thm}

\begin{proof}
($e_1$)
Let $x_0\in B$  satisfies the conditions of the theorem, then $(1+\epsilon) x_0 \not\in B$. Now we take
$E=\uparrow~ (1+~\epsilon) x_0$. $E$ is a compact set, so
by the strict separation theorem, cited in Section \ref{sectionp},
 there exists a  continuous
 linear functional $g$ such that $g(b) < (1+\epsilon)g(x_0)$  for all $b\in B$. Since B is bounded, so the function $g$ can be chosen such that $g(B) \leq 1$.

  Now, let $0<\delta<1$, by Lemma \ref{lem1}, there exists an $m\in B$
 satisfying $B\cap (m+K(g,\delta ,d)) = \{m\}$, $x_0 \sqsubseteq m$ and $x_0 \leq m$.
  Hence, $\delta(m-x_0) \leq g(m-x_0)d$. Therefore,  $\delta(m-x_0) \leq \epsilon g(x_0) d$ and so $\delta(m-x_0) \leq \epsilon  d$.
By Lemma \ref{lem2}, $B\cap int(\uparrow (m + K\verb|\| \{0\}))=
\varnothing$. Applying  the separation theorem, there exists a
continuous linear functional $f: C \rightarrow \overline{\Bbb
R}_+$  such that
 $f (b) \leq 1 < f(w)$ for all $b\in B$ and all $w \in int(\uparrow (m +  K\verb|\| \{0\}))$.
  By the continuity of $wd$-cone $C$, we have
$f (b) \leq f(w)$ for all $b\in B$ and all $w \in \twoheaduparrow (m +  K\verb|\| \{0\})$. By additive property(ii),
$f (b) \leq f(m+k)$ for all $b\in B$ and all $k \in K\verb|\| \{0\}$,
   and so
$f(b)\leq f(m)$ for all $b\in B$. Hence $\sup f(B) = f(m)$ and $f(m) <\infty$.

($e_2$) Let $f$ be a  continuous linear functional and let $0<\delta<1$ and $d\in C$. We consider the subcone
$K=~K(f,\delta ,d)$. By Lemmas  \ref{lem1} and \ref{lem2}, there exists an $m\in B$ such that
 $B\cap int( \uparrow (m+K(f,\delta ,d)\verb|\| \{0\})) =~\varnothing$.
 So by the separation theorem there exists a  continuous linear functional $h: C \rightarrow \overline{\Bbb R}_+$ satisfying $h(b) \leq h(m+c)$ for all $b\in B$ and  $c\in K(f,\delta ,d)$. This implies that $h$ attains its maximum. It follows that
$h(\delta c) \leq f(c) h(d)$. The number  $\delta$ and  the element $d$ can be taken to have $h(d)=\delta$. Thus  $0\leq h \leq f$ on a subcone of $C$.
\end{proof}
 It is well known that each norm on a linear space $X$ induces a metric on $X$.
 Note that, using the same method as in classical case,  each norm on a cone $C$ necessarily does not induce a metric on $C$. For example,
 the usual norm $\|x\| := x$ on  $\Bbb R_+$, does not induce a metric on  $\Bbb R_+$.
The cone order of each normed cone produce  a topology (named the Scott topology) on a normed cone. To know more about the relationship between this topology and concept of norm refer to \cite{selinger}.
\begin{cor}
Let $B$  be a nonempty convex bounded  Scott closed  set  in a  continuous normed cone such that for any scalar $\lambda >1$ and $x>0$, $x\ll \lambda x$.
 Then the results of Theorem \ref{asli} are still true.
\end{cor}
\begin{proof}
In a continuous normed cone, addition and scalar multiplication are Scott continuous
{\cite[Lemma  2.12]{selinger}},
so any continuous  normed cone  is a cancellative  $s$-cone with the  Scott topology. It is easy to see that every continuous  normed cone is a $bd$-cone.
Appealing to {\cite[Lemma 2.16]{selinger}}, $C$ has the additive property, thus assumptions of Theorem \ref{asli} hold and so we can conclude the desired statement.
\end{proof}
Let us illustrate the above theorem with some examples:
\begin{exm}
($f_1$)
Let $C = \overline{\Bbb R}_+^2$ and
$B = \{ (x,y)\in\overline{\Bbb R}_+^2  ; x+y \leq 1 \}$. Then $B$ is a
convex Scott closed  set which has no any maximum. It can be easily
checked that the $c$-support points  of $B$ is the set
$\{(x,y)\in~\overline{\Bbb R}_+^2 ; x+y = 1 \}$.

($f_2$)
For $d=(d_1, d_2,... )\in
\ell_1^+$,  the  set
$$B_d := \{x=(x_1, x_2, ... )\in  \ell_1^+ \ \ :  \ \ x\leq d\}$$
 is a bounded Scott closed set in $ \ell_1^+$ that has a maximum,
so  any  linear Scott continuous functional $f :~\ell_1^+ \rightarrow ~\overline{\Bbb R}_+$  takes its
supremum on $B_d$ at  the point  $d$. Observe that the set of
$c$-support points  of $B_d$ is
$$\{x\in \ell_1^+ \ : \ \exists f \in (\ell_1^+)^* \ \text{s.t.} \ f(x) = f(d)\}.$$
One can check that $\ell_{\infty}^+ \subset(\ell_1^+)^*$.
 Let $z$ belong to the following set,
$$D:= \{x=(x_1, x_2, ... )\in B_d \  : \ \text{for  some}\   i, \ x_i = d_i\}.$$  If we
take $a= (a_1, a_2, ...) \in \ell_{\infty}^+$, such that $a_i=0$  whenever
$z_i\neq d_i$, then $a$ is a $c$-support functional  and   $z$ is a $c$-support point for $B_d$.
Hence, $D$ is the set of  $c$-support points of $B_d$.
\end{exm}
\subsection{A fixed point result in $s$-cones}
What follows is an application of the Bishop-Phelps technique in some fixed point results.
Let $X$ be any space and $f$ a map of $X$, or of a subset
of $X$, into $X$. A point $x\in X$ is  called a fixed point for $f$  if $x = f(x)$.
The set of all fixed points of $f$  is denoted by $Fix(f)$.

Let $C$ be a $d$-cone. In {\cite[Theorem 2.1.19.]{abram}} the authors proved that every continuous function $f$ on $C$ has a least fixed point.
The property that ``in $d$-cones, every directed subset has a supremum" is applied in the proof of the  theorem. Since in $wd$-cones this property does not remain true  in general,  we establish a  fixed point result in $wd$-cones  by using the Bishop-Phelps technique.

\begin{thm}Let $C$ be a  $wd$-cone and $f: C \rightarrow \overline{\Bbb R}_+$ be a  continuous linear functional and  $0<\delta <1$ and $d\in C$. Suppose that $B$ is a bounded closed set in $C$ and  $T:C\rightarrow C$ is a mapping such that $f(B)\subset B$. Then the following assertion  holds.

If for each $x\in C$, $x\sqsubseteq Tx$, then there exists $m\in B$ such that $T(m) = m$.
\end{thm}
\begin{proof}
By applying Lemma \ref{lem1}, $B$ has a maximal element $m$. Now, using the assumption, $m\sqsubseteq T(m)$. The maximality of $m$  implies $T(m)=m$.
\end{proof}

\bigskip
\end{document}